\documentclass[10pt,a4paper]{amsart}
\usepackage{amsmath, amssymb, amsfonts, amsthm, float, stmaryrd, epsfig}
\usepackage[abbrev,alphabetic]{amsrefs}
\usepackage[pdfborder={0 0 0 [0 0 ]},bookmarksdepth=3, bookmarksopen=true]{hyperref}
\usepackage[latin1]{inputenc}
\usepackage[capitalize]{cleveref}
\usepackage{color}
\usepackage[ps,all,arc,rotate]{xy}
\usepackage{bbm} %allows for \mathbb{1}
\usepackage{stackengine,scalerel}
\stackMath

\hypersetup{
    colorlinks=true,
    linkcolor=black,
    citecolor=black,
    filecolor=black,
    urlcolor=black,
}

\numberwithin{figure}{section}
\numberwithin{table}{section}

\theoremstyle{plain}
\newtheorem{thm}{Theorem}[section]
\crefname{thm}{Theorem}{Theorems}
\newtheorem*{prop*}{Proposition}
\newtheorem*{thm*}{Theorem}
\newtheorem{prop}[thm]{Proposition}
\crefname{prop}{Proposition}{Propositions}
\newtheorem{lem}[thm]{Lemma}
\crefname{lem}{Lemma}{Lemmata}
\newtheorem{cor}[thm]{Corollary}
\crefname{cor}{Corollary}{Corollaries}

\crefname{conj}{Conjecture}{Conjectures}
\crefname{equation}{Equation}{Equations}

\theoremstyle{definition}

\newtheorem{dfn}[thm]{Definition}
\newtheorem*{dfn*}{Definition}

\theoremstyle{remark}
\newtheorem{rmk}[thm]{Remark}

\newtheoremstyle{maintheorem}{}{}{\itshape}{}{\bfseries}{}{.5em}{#1 \!\thmnote{\ #3}.}
\theoremstyle{maintheorem}

\makeatletter
\let\c@figure\c@thm
\let\c@table\c@thm
\makeatother
\crefname{figure}{Figure}{Figures}
\crefname{table}{Table}{Tables}

\newcommand{\SL}{\operatorname{SL}}

\newcommand{\PGL}{\operatorname{PGL}}
\newcommand{\PSL}{\operatorname{PSL}}

\newcommand{\id}{\operatorname{id}}
\newcommand{\im}{\operatorname{im}}
\newcommand{\supp}{\operatorname{supp}}

\newcommand{\typeF}[1]{\mathtt{F}_{#1}}
\newcommand{\typeFP}[1]{\mathtt{FP}_{#1}}
\newcommand{\typeFF}[1]{\mathtt{FF}_{#1}}

\def\C{\mathbb{C}}
\def\R{\mathbb{R}}

\def\Z{\mathbb{Z}}
\def\Q{\mathbb{Q}}

\def\1{\mathbbm{1}}

\def\F{\mathbb{F}}

\def\s-{\smallsetminus}

\def\iff{if and only if }

\newcommand{\PD}[2]{\operatorname{PD}^{#1}_{#2}}

\newcommand{\Hom}{\operatorname{Hom}}

\newcommand{\rk}{\operatorname{rk}}

\binoppenalty=\maxdimen
\relpenalty=\maxdimen

\newcounter{dawidcomments}
\newcommand{\dawid}[1]{\textbf{\color{red}(D\arabic{dawidcomments})} \marginpar{\scriptsize\raggedright\textbf{\color{red}(D\arabic{dawidcomments})Dawid:} #1}
\addtocounter{dawidcomments}{1}}

\author{Dawid Kielak}
\address{Dawid Kielak, Mathematical Institute, University of Oxford, Andrew Wiles Building,
Radcliffe Observatory Quarter,
Woodstock Road,
Oxford
OX2 6GG,
United Kingdom}
\email{kielak@maths.ox.ac.uk}

\author{Peter Kropholler}
\address{Peter Kropholler, Mathematical Sciences, University of Southampton,
Southampton SO17 1BJ, United Kingdom}
\email{P.H.Kropholler@southampton.ac.uk}

\title{Isoperimetric inequalities for Poincar\'e duality groups}

\begin{document}

\begin{abstract}
We show that every oriented $n$-dimensional Poincar\'e duality group over a $*$-ring $R$ is amenable or satisfies a linear homological isoperimetric inequality in dimension $n-1$. As an application, we prove the Tits alternative for such groups when $n=2$. We then deduce a new proof of the fact that when $n=2$ and  $R = \Z$ then the group in question is a surface group.
\end{abstract}

\maketitle
\section{Introduction}

The classification of $2$-dimensional Poincar\'e duality groups over $\Z$ is a milestone in group theory built on a careful case analysis of Eckmann's student M\"uller \cite{Mueller1979} and the subsequent joint work of Eckmann and M\"uller \cite{EckmannMueller1980}. The latter paper gives a characterization of such groups relying only on the Poincar\'e duality property and the assumption of positive first Betti number. A remarkable argument of Linnell shows that positive first Betti number is also a consequence of Poincar\'e duality, see \cites{EckmannLinnell1982,EckmannLinnell1983}.
Linnell's insight calls the Bass Conjecture, already established for groups of low cohomological dimension in various cases, and with positive first Betti number established it then follows quickly that $2$-dimensional Poincar\'e duality groups are HNN-extensions over a base that is free of finite rank. In sum, there is the following classfication theorem:

\begin{thm}
\label{classical thm}
Let $G$ be a group. Then the following are equivalent.
\begin{enumerate}
\item
$G$ is a $2$-dimensional Poincar\'e duality group over $\Z$.
\item
$G$ has a presentation of one of the following two kinds:
\begin{enumerate}
\item $G\cong\langle g_1\,\dots,g_{2n};\ \prod_{i=1}^n[g_{2i-1},g_{2i}]=1\rangle$ for some $n\geqslant 1$, or
\item $G\cong\langle g_1,\dots,g_n;\ \prod_{i=1}^n g_i^2=1\rangle$ for some $n\geqslant 2$.
\end{enumerate}
\item
$G$ is isomorphic to the fundamental group of a closed surface of genus $\geqslant 1$.
\end{enumerate}
\end{thm}
Of course the equivalence of (2) and (3) and the implication (3)$\Rightarrow$(1) are classical. The Eckmann--M\"uller strategy for proving the implication (1)$\Rightarrow$(2)  involves an intricate analysis of cases which runs, in spirit, parallel to the classical case analysis required to prove (3)$\Rightarrow$(2).

Subsequently it has been of interest to study groups which satisfy Poincar\'e duality over a field. For example, if $G$ is a $3$-dimensional Poincar\'e duality group with non-trivial centre $Z$ (as arises for the fundamental group of a closed Seifert fibered $3$-manifold) then one might be able to deduce, in favourable circumstances, that $G/Z$ satisfies Poincar\'e duality over $\Q$ even though one cannot expect the result to hold over $\Z$ because, for example, $G/Z$ might not be torsion-free. Therefore an important next step is found in the work of Bowditch \cite{Bowditch2004} in which a much more general result is proved of which the following is a special case.

\begin{thm}
Suppose that $G$ is a $2$-dimensional Poincar\'e duality group over some field $\F$. Then $G$ is a virtual surface group.
\end{thm}

Bowditch necessarily has to deal with some formidable difficulties and it is a considerable achievement even to be able to exclude infinite torsion groups or Tarski monsters from the story, and his 51 page paper is no light reading matter. 
We mention that Bowditch's work has been extended by Kapovich and Kleiner \cite{KapovichKleiner2004}.
Looking back from the present day perspective, there is some attraction to researching if modern methods could be used to simplify or provide new insights.

In this paper we provide a very short and quick argument for the following. While this does not come close to Bowditch's definitive result for commutative rings, given the modern day interest in both the class of hyperbolic groups and the class of amenable groups it may be of interest to see a succinct and direct argument even in the commutative case.

\begin{thm}
\label{our thm}
Suppose that $G$ is a $2$-dimensional Poincar\'e duality group over a commutative ring $R$, or an oriented $2$-dimensional Poincar\'e duality group over a $*$-ring $R$. Then either $G$ is non-elementary hyperbolic or $G$ is amenable.
\end{thm}
(See \cref{* def} for the definition of a $*$-ring.)

Of course, the importance of this is that in the case of hyperbolic or amenable groups one has a chance of rather more direct strategies of completing the original Eckmann--M\"uller analysis and showing directly that these groups are surface groups, at least when $R = \Z$. In this article we focus on the above theorem and we deduce it from a result about Poincar\'e duality groups of arbitrary dimension that concerns
homological isoperimetric inequalities. We then proceed to outline a relatively short proof of a variant of \cref{classical thm} following from \cref{our thm}.

\subsection*{Acknowledgements}
The authors would like to thank Yuri Santos R\^{e}go and Stefan Witzel for a helpful discussion concerning arithmetic groups, and Robert Kropholler for useful observations regarding the formulation of Proposition 4.1; they would also like to thank the referee for significantly improving the presentation and accuracy of the article.

The first author was partly supported by a grant from the German Science Foundation (DFG) within the Priority Programme SPP2026 `Geometry at Infinity'. This work has received
funding from the European Research Council (ERC) under the European Union's Horizon 2020 research and innovation programme, Grant
agreement No.\ 850930.

\section{Homological isoperimetric inequalities}

Throughout, $G$ denotes a discrete group, $R$ denotes an associative  unital ring with at least $2$ elements, and $RG$ denotes the usual group ring of $G$ with coefficients in $R$. %, and $D$ denotes some non-zero $R G$-module.
All modules are left modules.%, unless specified otherwise.
%and unspecified tensoring takes place over $R$.

%Note that every left $RG$-module has also the structure of a right $RG$-module, where the right multiplication by $g$ is defined to coincide with the left multiplication by $g^{-1}$, and the $R$ actions on both sides are the same.

%We use $DG$ to denote the $RG$-module
%\[\bigoplus_{g \in G} D = \{ \sum_{g \in F} d_g g \mid F \subseteq G \textrm{ is finite} \}\]
%where the $R G$ action is induced by $g . d_h h = (g .d_h) gh$.
%Note that $RG$ is also a right $R$-module, and so we have $D G = R G \otimes R$ as $R$-modules. The action of $G$ on $R G \otimes D$ is the diagonal action.

\begin{dfn}
%Let $n$ denote the cohomological dimension of $G$, and assume that $n \in \N$.
A projective resolution $C_\bullet = (C_k, \partial_k)$ of the trivial $RG$-module $R$ is \emph{$n$-admissible} \iff
 $C_{n}$ and $C_{n+1}$ are finitely generated free $RG$-modules equipped with free bases.
 
 Since the bases of $C_n$ and $C_{n+1}$ are fixed, we will treat $n$ and $n+1$ chains as vectors and the differential $\partial_{n+1} \colon C_{n+1} \to C_n$ as a matrix over $RG$.
\end{dfn}

\iffalse

\begin{dfn}
%Let $n$ denote the cohomological dimension of $G$, and assume that $n \in \N$.
A resolution $C_\bullet$ of the trivial $RG$-module $R$ is \emph{of type $\typeFF{n}$} \iff
 $C_k$ is a finitely generated free $RG$-module for every $k \leqslant n$.
\end{dfn}

\begin{dfn}[Admissible resolution]
%Let $n$ denote the cohomological dimension of $G$, and assume that $n \in \N$.
A resolution $C_\bullet$ of the trivial $RG$-module $R$ is called \emph{$n$-admissible} \iff
\begin{itemize}
 \item $C_k$ is a finitely generated free $RG$-module for every $k \leqslant n$, and
% \item $C_{n+1}$ is a free $RG$-module, and
 \item \dawid{I am not sure if we need this}we have $C_0 = RG$ and $C_1 = R G^{\vert S \vert}$ where $S = \{s_1, \dots, s_m\}$ is some finite generating set of $G$ and the differential $C_1 \to C_0$ is given by the $\vert S \vert \times 1$-matrix $(1-s_i)_i$.
\end{itemize}
\end{dfn}
\fi

Note that every group of type $\typeFP{n+1}$ over $R$ admits an $n$-admissible resolution of $R$, (for details see \cite{Brown1982}*{VIII Prop.4.3}).

\begin{dfn}
 For an element $x \in R G$ we define $\vert x \vert$ to be the cardinality of the support of $x$. For a finite matrix $A = (a_{ij})_{i,j}$ over $RG$ we define $\vert A \vert = \sum_{i,j} \vert a_{ij} \vert$; we treat vectors as special cases of matrices.
 %We extend the definition to infinite matrices, provided that they have only finitely many non-zero entries.
\end{dfn}

\begin{dfn}[Homological isoperimetric inequality]
 Let $C_\bullet$ be an $n$-admissible projective resolution of the trivial $RG$-module $R$.
  We say that $C_\bullet$ satisfies a \emph{linear isoperimetric inequality in dimension $n$} \iff there exists $\kappa \geqslant 0$ such that for every  $n$-boundary $\gamma \in C_n$ there exists an $(n+1)$-chain $d
\in C_{n+1}$ such that
 \[
  \partial d = \gamma \textrm{ and }\vert d \vert \leqslant \kappa \vert \gamma \vert
 \]
%where $\partial$ is the $(n+1)^{st}$ differential of $C_\bullet$.

We say that a group $G$ satisfies a \emph{linear $R$-homological isoperimetric inequality in dimension $n$} \iff it admits an $n$-admissible resolution of the trivial $RG$-module $R$ which satisfies a linear isoperimetric inequality in dimension $n$.
\end{dfn}

The following lemma shows that this property is independent of the resolution used.

\begin{lem}
\label{independence}
Let $C_\bullet$ and $C'_\bullet$ be two $n$-admissible resolutions of the trivial $RG$-module $R$. If $C_\bullet$ satisfies a linear isoperimetric inequality in dimension $n$, then so does $C'_\bullet$.
\end{lem}
\begin{proof}
Since $C_\bullet$ and $C'_\bullet$ are both resolutions of $R$, there exist $RG$-chain maps $\xi \colon C_\bullet \to C'_\bullet$ and $\zeta \colon C'_\bullet \to C_\bullet$ such that $\xi \circ \zeta$ is chain homotopic to the identity via an $RG$-chain homotopy $h$.

Let $\gamma \in C'_{n}$ be a boundary. Since $\zeta$ is a chain map, $\zeta(\gamma)$ is also a boundary. Moreover, since $C'_{n}$ and $C_{n}$ are finitely generated free $RG$-modules, the map $\zeta_n \colon C'_n \to C_n$ can be represented by a finite matrix $Z$ over $R G$, and we immediately see that
\[
| \zeta(\gamma)| \leqslant |Z| | \gamma |
\]
Now $C_\bullet$ satisfies a linear isoperimetric inequality in dimension $n$ by assumption. Let $\kappa$ be the constant given by the definition, and let $d \in C_{n+1}$ be such that
\[
\partial d = \zeta(\gamma) \textrm{ and } \vert d \vert \leqslant \kappa \vert \zeta(\gamma) \vert \leqslant \kappa |Z| | \gamma |
\]
Now let $d' = \xi_{n+1}(d) \in C'_{n+1}$. Observe that
\[
\partial d' = \partial \xi(d) = \xi (\partial d) = \xi \circ \zeta(\gamma) = \gamma - \partial h(\gamma)
\]
Observe also that $\xi_{n}$ and $h_{n}$ %, and $\partial_n \colon C_n \to C'_n$
are represented by finite matrices over $R G$; let us denote the matrices by $X$ and $H$ %, and $P$
 respectively. We have
\[
| d' +  h(\gamma)| \leqslant |  d'| + | h(\gamma)| \leqslant |X| |d| + |H| |\gamma| \leqslant \big(\kappa|X||Z| + |H|\big) |\gamma|
\]
which proves the claim.
\end{proof}

\iffalse
\begin{rmk}
\label{independence rmk}
Observe that in the above lemma we only use the fact that the chains $C_{n-1}, C_n, C'_{n-1}$, and $C_n$ are free and finitely generated; the other chains play no role.
\end{rmk}
\fi

Note that %, taking $D = R$ as an $RG$-module,
our definition is \emph{not} the usual notion of homological isoperimetric inequality:  Mineyev~\cites{Mineyev2000,Mineyev2002}, and Groves--Manning \cite{GrovesManning2008}*{Section 2.7} consider only the cases of $R \in \{\Z, \Q, \R, \C\}$, and use the usual absolute value on $R$ and the resulting $L^1$-norm on $R G$. Gersten~\cite{Gersten1998}*{Section 13} however does consider all rings $R$ endowed with an abelian group norm. Our notion of $\vert \cdot \vert$ is such a norm, since such norms are not required to be $\Z$-multiplicative, and hence we are working with a particular example of what Gersten allows.

Mineyev~\cite{Mineyev2000} and Lang~\cite{Lang2000} showed that every hyperbolic group satisfies homological linear isoperimetric inequalities in all dimensions when the coefficients lie in $\R$ or $\Z$, respectively (note that Lang uses yet another definition of isoperimetric inequalities). Thus, the reader is invited to think of the linear inequalities as some (very weak) form of negative curvature.

\smallskip

Before proceeding further, we will need a simple observation about non-amenable groups.

\iffalse
Note that this is consistent with our previous use of $\vert \cdot \vert$. This is also the $L^1$-norm on $RG$ induced by the abelian group norm on $R$ given by
\[
 r \mapsto \left\{ \begin{array}{ll} 1 & \textrm{if } r \neq 0 \\ 0 & \textrm{otherwise} \end{array} \right.
\]
\fi

\begin{lem}
\label{amen}
Let $G$ be a group with a finite generating set $\mathcal S = \{s_1, \dots, s_m\}$. Let $\partial \colon R G \to R G^m$ be given by the $m \times 1$ matrix $(1-r_i s_i)_i$ where $r_i \in R$ for every $i$. If $G$ is not amenable, then there exists $\kappa \geqslant 0$ such that for every $\gamma \in \im \partial$ and every $d \in R G$ with $\partial(d) = \gamma$ we have
\[
 \vert d \vert \leqslant \kappa \vert \gamma \vert
\]
\end{lem}
\begin{proof}
 Assume that $G$ is not amenable; the (negation of the) F\o{}lner criterion tells us that there exists $\epsilon > 0$ such that for every finite subset $F \subseteq G$ we have $\vert B(F) \vert \geqslant \epsilon \vert F \vert$ where
 \[
 B(F) = \big\{ g \in F \mid g{s_i}^{-1} \not\in F \textrm{ for some } i \big\}
 \]
 ($B(F)$ is the \emph{boundary} of $F$).

 Now, take $d \in R G$ and set $\gamma = \partial(d)$, as in the statement. Let $F$ be the support of $d$. Every element of $B(F)$ appears in the support of at least one of the entries of $\partial(d)$, and therefore we have $\vert \gamma \vert \geqslant \epsilon \vert d \vert$. Setting $\kappa = \frac 1 \epsilon$ finishes the proof.
\end{proof}

\section{Poincar\'e duality groups}

The following is an extension of the definition of Poincar\'e duality groups due to Johnson--Wall \cite{JohnsonWall1972}; note that we do not assume $G$ to be finitely presented.

\begin{dfn}
A group $G$ is called a \emph{Poincar\'e duality} group of dimension $n \geqslant 1$ over $R$, or a $\PD n R$ group, \iff
$G$ is of type $\typeFP{}$ over $R$ and we have
\[
H^k(G; RG) = \left\{ \begin{array}{ll} R & \textrm{ if } k=n \\ 0 & \textrm{ otherwise} \end{array} \right.
\]
where the equality holds as  $R$-modules.

The $\PD n R$ group is \emph{oriented} \iff $H^k(G; RG) = R$ is the trivial $RG$-module.
\end{dfn}

It is classical (and can be proven along the lines of \cite{Brown1982}*{Proposition 6.7}) that the dimension $n$ coincides with the cohomological dimension of $G$ over $R$.

\begin{dfn}
\label{* def}
The ring $R$ is a \emph{$*$-ring} \iff it comes equipped with an involutive antiautomorphism $R \to R, r \mapsto r^*$.
\end{dfn}

Every commutative ring is a $*$-ring with $*$ being the identity operation. In the case of $\C$ one usually takes $*$ to be the complex conjugation. A more interesting example is provided by quaternions, where $*$ is the quaternion conjugation.

Note that if $R$ is a $*$-ring then so is $RG$, with
\[
\big( \sum_g r_g g \big)^* = \sum_g {r_g}^* g^{-1}
\]

%Note that $\PD n R$ groups are always finitely generated. This can be proven in exactly the same way as \cite{BieriEckmann1973}*{Theorem 4.6}.

\begin{thm}
\label{main thm}
Let $n$ be a positive integer, and let $G$ be a $\PD n R$ group. Suppose that additionally one of the following holds:
\begin{enumerate}
\item $R$ is a $*$-ring and $G$ is orientable;
\label{case 1}
\item $R$ is commutative.
\label{case 2}
\end{enumerate}
If $G$ is not amenable, then $G$ satisfies a linear $R$-homological isoperimetric inequality in dimension $n-1$.
\end{thm}
\begin{proof}
Let $G$ be a  $\PD n R$ group. Since $G$ is in particular of type $\typeFP{1}$ over $R$, it is generated by a finite set, say $\mathcal S = \{s_1,\dots,s_m\}$.
Duality guarantees that the cohomological dimension of $G$ over $R$ is $n$. 

If $n=1$ then  Dunwoody~\cite{Dunwoody1979} shows that $G$ is the fundamental group of a finite graph of finite groups, and hence the group $G$ is virtually free (by 
\cite{Serre2003}*{II.2.6 Corollary}). 
Dunwoody states his theorem for a commutative ring $R$, but his proof works also for non-commutative ones, as he remarks at the end of the paper. Let $H$ be a free subgroup of finite index in $G$ -- this inherits the $\PD{1}{R}$ property by \cite{JohnsonWall1972}*{Theorem 2}. Note that Johnson--Wall work with $R=\Z$, but their proof readily adapts to arbitrary $R$. Every free group which is a $\PD 1 R$ group is infinite cyclic, as can be seen by looking at the first cohomology with free module coefficients. We conclude that $G$ is virtually cyclic and hence amenable when $n=1$.

Now assume that $n>1$ and that $G$ is not amenable. Since the cohomological dimension of $G$ is $n$, there is a projective resolution
$$0\to P_n\to P_{n-1}\to\dots\to P_1\to P_0\to R\to 0$$ where
 $P_0=RG$, $P_1=RG^m$ (row vectors of length $m$), and the map $P_1\to P_0$ is given by right multiplication by the $m\times 1$-matrix (or column vector) $(1-s_1,\dots,1-s_m)^T$.

%We now change the resolution $P_\bullet$ in the following manner: we set $C_{k} = P_k$ for $k<n$, and replace $P_k$ by a finitely generated free module $C_n$ containing $P_n$. We now add an extra projective module $C_{n+1}$ to make the chain into a resolution. Note that the resulting resolution is of type $\typeFF{n}$.

For a left $RG$-module $M$ we define $M'$ to be the left $RG$-module $\Hom_{RG}(M,RG)$ on which $x \in RG$ acts by the formula
\[x\theta (m)=\theta( x^*m )\]
(When $R$ is commutative, we take $*$ to be the identity on $R$, and so $(rg)^* = r g^{-1}$.)
In this way, the operation ${(-)}'$ defines a contravariant functor from left $RG$-modules to left $RG$-modules.
Since $G$ satisfies Poincar\'e duality over $R$, we obtain the exact sequence
\[\tag{$\dagger$} \label{dual resolution} 0\to P_0'\to P_{1}'\to\dots\to P_{n-1}'\to P_n'\to D\to 0\]
on applying the functor ${(-)}'$ to the resolution above, where $D = H^n(G;RG)$ (see \cite{JohnsonWall1972}, again adapting the argument to a general $R$).

Suppose that we are in case \eqref{case 1}. Since $G$ is oriented, we have $D = R$ as an $RG$-module. Hence, the resolution \eqref{dual resolution} is an
$(n-1)$-admissible resolution of the trivial $RG$-module $R$.
It also clearly satisfies a linear isoperimetric inequality in dimension $n-1$ by \cref{amen}, which proves the claim.

\smallskip
Now suppose that we are in the (more involved) case \eqref{case 2}. Since $D \cong R$ as an $R$-module, it is immediate that there exists a group homomorphism $\rho \colon G \to R$ such that the $G$-action on $D$ is given by $gr = \rho(g) r$. Let $E$ denote $R$ endowed with the $G$-action $gr = \rho(g^{-1}) r$. We now see that the diagonal action of $G$ on $D \otimes_R E$ is trivial, as
\[
g (r \otimes r') = gr \otimes gr' = \rho(g)r \otimes \rho(g^{-1})r' = \rho(g^{-1})\rho(g)r\otimes r' = r \otimes r'
\]
Hence $D \otimes E = R$, the trivial $RG$-module.

We tensor \eqref{dual resolution} with $E$ over $R$ (and tensor the differentials with $\id_E$) and obtain a chain complex $C_\bullet$ where the module structure on the chains is given by diagonal action. Since $E \cong R$ as an $R$-module, it is immediate that $C_\bullet$ is exact.

We have a classical isomorphism of $RG$-modules $R G \to R G \otimes_R E$ induced by $g \mapsto g \otimes \rho(g^{-1})$ and with inverse induced by $g \otimes 1 \mapsto g \rho(g)$. The existence of such an isomorphism implies that
 $C_k$ is a projective module for every $k$; it tells us also that $C_n = P_0' \otimes_R E$ and $C_{n-1} = P_1' \otimes_R E$ are free $RG$-modules. Thus, $C_\bullet$ is an $(n-1)$-admissible projective resolution of the trivial module $R$.
Direct computation shows that $\partial \colon C_n \to C_{n-1}$ is equal to the column vector
\[(1-\rho(s_1)s_1,\dots,1-\rho(s_m)s_m)^T\]
 and so $C_\bullet$ satisfies a linear isoperimetric inequality in dimension $n-1$ by \cref{amen}.
\end{proof}

\iffalse
\begin{rmk}
\label{removing cells}
 Note that we can in fact remove all but finitely many $G$-orbits of $2$-cells from $X$, and obtain the same conclusion, since the images of the maps $\zeta_0$ and $h_{n-1}$ lie in a finitely generated submodule of $C_2$, and so there exist finitely many $G$-orbits of $2$-cells in $X$ such that the image is supported on the free $R G$-module generated by these $2$-cells.
\end{rmk}
\fi

\section{Dimension $2$}

The following result is at the heart of our arguments. Here, the coefficient ring $R$ can be any non-zero ring. In the case $R=\Z$ this result also follows from work of Gersten \cite{Gersten1996}*{Theorem 5.2}.

\begin{prop}
\label{linear ineq}
Let $G$ be a group of type $\typeFP{2}$ over $R$ which satisfies a linear $R$-homological isoperimetric inequality in dimension $1$. Then $G$ is Gromov hyperbolic.
\end{prop}
\begin{proof}
The proof is essentially identical to the Bridson--Haefliger proof of the corresponding fact for the usual, homotopic isoperimetric inequality -- see \cite{BridsonHaefliger1999}*{Chapter III.H Theorem 2.9}. We sketch it here for the convenience of the reader.

Let $X$ be a $2$-dimensional $G$-CW-complex on which $G$ acts freely and cocompactly and which is $1$-acyclic over $R$ (meaning that $H_1(X;R)=0$). Here we use the terminology \emph{$G$-CW-complex} as defined by L\"uck \cite{Lueck2002}*{Definition 1.25}: this is same notion as \emph{admissible $G$-complex} in the sense of Brown 
\cite{Brown1982}*{Chapter IX \S10}. To construct such an $X$ choose a presentation $2$-complex of $G$ based on bouquet of finitely many circles and let $X$ be a subcomplex of its universal cover that includes the $1$-skeleton and sufficiently but finitely many orbits of the $2$-cells to ensure that $H_1(X;R)=0$. It is here that we require $G$ to be of type $\typeFP{2}$ over $R$.
Note that the $1$-skeleton of $X$ is a Cayley graph for $G$ and the conclusion that $G$ is Gromov hyperbolic will follow from a proof that this $1$-skeleton with its combinatorial metric is Gromov hyperbolic.

We continue by fixing a constant $N$ to be the maximal number of edges contained in the image of the attaching map of any $2$-cell in $X$. Since the action of $G$ on $X$ is cocompact, the number $N$ is a well-defined integer. %If this integer is zero, we set $N=1$ instead.

The cellular chain complex of $X$ can easily be extended to a $1$-admissible projective resolution of $R$, which in turn must satisfy a linear isoperimetric inequality in dimension $1$ by \cref{independence}.
The isoperimetric inequality gives us a constant $\kappa$ such that every $1$-cycle in $X$ supported on $l$ edges is the boundary of a $2$-chain supported on at most $\kappa l$ faces. For convenience, we take $\kappa$ to be an integer.
We set $k =  \kappa N^2 + 1$ and $m =  \kappa N$.

We aim at showing that for some $n$, every geodesic triangle in $X$ is $(n+1)$-slim (for this notion, as well as other background information on Gromov hyperbolic spaces, we refer the reader to \cite{BridsonHaefliger1999}).
We argue by contradiction, and let $\Delta$ be a geodesic triangle in $X$ which is not $(n+1)$-slim, which amounts to saying that there exists a vertex $v$ on one of the sides of $\Delta$ which does not lie in the $n$-neighbourhood of the other two sides. We may assume that $n > 6k$.

The first part of the proof of \cite{BridsonHaefliger1999}*{Chapter III.H Theorem 2.9} shows that by either `cutting the corners' or `cutting a corner and the opposite edge' of $\Delta$ we arrive at one of the following situations (see \cref{figure}).

\begin{figure}[h]
 \includegraphics[scale=0.35]{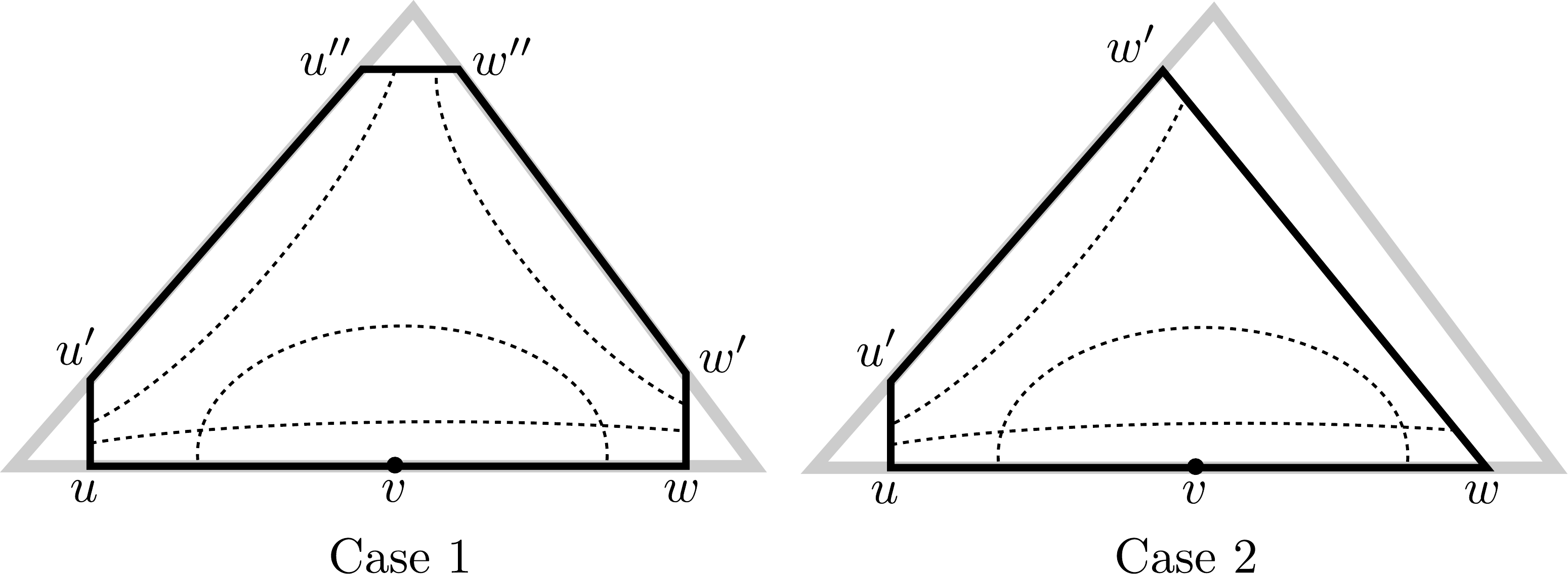}
\caption{The triangle $\Delta$ and the polygon $H$}
\label{figure}
 \end{figure}

\smallskip
\noindent \textbf{Case 1:}
There exists a geodesic hexagon $H$ in $X$ with vertices (written cyclically) $u'',u',u,w,w',w''$, such that (denoting the combinatorial distance in the $1$-skeleton of $X$ by $d$):
\begin{itemize}
 \item $d(u,u') = d(w,w') = d(u'',w'') = 2k$, and
 \item the $(k-1)$-neighbourhoods of the segments $[u,w], [u',u'']$, and $[w',w'']$ are pairwise disjoint, and
 \item there exists a point $v \in [u,w]$ such that the ball around $v$ of radius $n-k+1$ is disjoint from the $(k-1)$-neighbourhoods of $[w',w'']$ and $[u'',u']$.
\end{itemize}
We let $\alpha, \beta$, and $\gamma$ denote the lengths of the segments $[u,w], [u',u'']$, and $[w',w'']$, respectively. %We will abuse notation by referring to the edges as $\alpha, \beta$ and $\gamma$ as well.

Orienting the edges of $H$ in a coherent manner, and putting the weight $1$ on every edge we obtain an $R$-cycle $h$ (note that the ring $R$ plays no role here). Since $X$ is $1$-acyclic, the cycle $h$ is a boundary. Consider a $2$-chain $p$ with $\partial p = h$ which minimises the value $\vert p \vert$, and let $P$ denote the support of $p$.

Set $h_{[u,w]}$ to be the $1$-chain obtained from putting the weight $1$ only on edges in $[u,w]$.
Let $D_0$ denote the set of all the $2$-cells in $P$ whose boundary contains an edge from the segment $[u,w]$. Since every $2$-cell has at most $N$ faces, and since every edge in the segment $[u,w]$ has to appear in the boundary of some cell in $P$, we conclude that $\vert D_0 \vert \geqslant \alpha/N$.
Let $p_0$ denote the $2$-chain obtained from $p$ by setting the coefficient of $2$-cells not lying in $D_0$ to $0$. Observe that the $1$-chain
\[
 \partial p_0 - h_{[u,w]}
\]
contains in its support an edge path connecting $u$ to $w$, which does not contain any edges of the segment $[u,w]$.

We will now recursively define pairwise-disjoint subsets $D_1, \dots, D_{m-1}$ of $P$ and $2$-chains $p_1, \dots, p_{m-1}$ as follows: suppose that for some $i$ the $1$-chain
\[
 q_i = \partial (\sum_{j=0}^i p_j ) - h_{[u,v]}
\]
contains an edge path connecting $u$ to $w$ (we know this to hold for $i=0$). We set $D_{i+1}$ to be subset of $P \s- \bigcup_{j=0}^i D_i$ containing all $2$-cells whose boundaries contain at least one edge of $\supp q_i$. We set $p_{i+1}$ to be $2$-chain obtained from $p$ by restricting the support to $D_{i+1}$.

One key observation is that the faces of $D_i$ lie entirely in the $N(i+1)$-neighbourhood of the segment $[u,w]$, and therefore the $2$-cells in $D_{[u,w]} = \bigcup_{i=0}^{m-1} D_i$ lie in the $Nm$-neighbourhood. But $Nm =  \kappa N^2 = k - 1$, and so the union lies in the $(k-1)$-neighbourhood.

Another is that
the boundaries of faces in $D_i$ have to include every edge in the support of $q_i$, with the possible exception of the edges lying on the segments $[u,u']$ and $[w,w']$. But there are at most $2k-2$ such edges, and so
$\vert D_i \vert \geqslant \frac {\alpha - 2k +2}N$, and thus
\[
 \left\vert D_{[u,w]} \right\vert \geqslant \frac {m (\alpha-2k+2)} N = \kappa (\alpha-2k+2)
\]

We repeat the argument for the other segments, and obtain a subset $D_{[w',w'']} \subseteq P$ of cardinality at least $\frac {m (\beta -2k -2)} N$ whose elements lie in the $(k-1)$-neighbourhood of the segment $[u',u'']$, and a subset $D_{[u'',u']} \subseteq P$ of cardinality at least $\frac {m (\gamma-2k-2)} N$ whose elements lie in the $(k-1)$-neighbourhood of the segment $[w',w'']$. Since the three $(k-1)$-neighbourhoods are disjoint, we conclude that
\[
 \vert P \vert \geqslant \kappa (\alpha + \beta + \gamma -6k + 6)
\]

We are now going to find further $2$-cells in $P$ not contained in the union \[D_{[u,w]} \cup D_{[w',w'']} \cup D_{[u'',u']}\]
Note that every edge in $\supp q_{m-1}$ is of distance at most $k-1$ from the segment $[u,w]$. Consider a function taking every edge in $\supp q_{m-1}$ to its closest vertex on $[u,w]$ (if there is more than one closest vertex, we choose one). Now the support of $q_{m-1}$ contains a path $f$ connecting $u$ to $w$, and two adjacent edges in $f$ are sent to vertices at most $2k$ apart (since $[u,w]$ is a geodesic segment), and therefore the $2k$-neighbourhood of $v$ contains some edges in $\supp f$. This implies that the $n-2k$-neighbourhood of $v$ contains at least $2(n-4k)$ edges of $\supp f$. But each such edge must lie in the boundary of some $2$-cell in $P \s- (D_{[u,w]} \cup D_{[w',w'']} \cup D_{[u'',u']})$, and we conclude that
\[
 \vert P \vert \geqslant \kappa (\alpha + \beta + \gamma -6k + 6) + \frac {2n - 8k}N
\]
Now we observe that $\vert h \vert = \alpha + \beta + \gamma + 6k$, and the linear isoperimetric inequality yields
\[
 \kappa(\alpha + \beta + \gamma + 6k) \geqslant \vert P \vert \geqslant \kappa (\alpha + \beta + \gamma -6k + 6) + \frac {2n - 8k}N
\]
which is impossible for large $n$, as $k$ and $N$ do not depend on $n$.
\iffalse
with vertices $p,q$ and $r$ in $X$, and assume that $\Delta$ is not $(n+1)$-slim. By definition, this means that there exists a point $a$ in the geodesic $[p,q]$ connecting $p$ to $q$, such that $d(a, [q,r] \cup [r,p]) > n+1$, where $d$ denotes the combinatorial distance in the $1$-skeleton of $X$. %Let $\Delta$ have the smallest possible perimeter among such triangles.
We replace $a$ by an adjacent vertex $v$, and conclude that
\[
 d(v, [q,r] \cup [r,p]) > n
\]
For two points $x$ and $y$ on the geodesic $[p,q]$, we will write $[x,y]$ for the segment of $[p,q]$ connecting $x$ to $y$; we will use analogous notation for the other geodesics.

We now set $k = 3 \kappa N^2$ and $m = 3 \kappa N$. We will assume that $n > 6k$. There are two cases to consider (up to obvious symmetries): either $[p,v]$ is disjoint from the $4k$-neighbourhood of $[q,r]$ and $[v,q]$ is disjoint from the $4k$-neighbourhood of $[r,p]$ (this will be \emph{case 1}), or there exist points $w \in [v,q]$ and $w' \in [r,p]$ with $d(w,w') \leqslant 4k$ (this will be \emph{case 2}).

\smallskip We start by examining case 1. Let $u \in [p,v]$ be the vertex farthest from $v$ such that $d(u, [r,p])=2k$; similarly, let $w \in [v,q]$ be the vertex farthest from $v$ such that $d(w, [q,r])=2k$.
\dawid{add a figure}
Let $u' \in [p,r]$ be the closest vertex to $r$ with $d(u',u) = 2k$. Similarly, let $w' \in [q,r]$ be the closest vertex to $r$ with $d(w',w) = 2k$. We let $u'' \in [u',r]$ denote the vertex closest to $r$ such that
\fi

\smallskip
\noindent \textbf{Case 2:}
There exists a geodesic quadrilateral $H$ in $X$ with vertices (written cyclically) $u',u,w,w'$, such that
\begin{itemize}
 \item $d(u,u') =  2k, d(w,w') = 4k$, and
 \item the $(k-1)$-neighbourhoods of the segments $[u,w]$ and $[u',w]$ are disjoint, and
 \item there exists a point $v \in [u,w]$ such that the ball around $v$ of radius $n-2k$ is disjoint from the $(k-1)$-neighbourhood of $[w',u']$.
\end{itemize}
We let $\alpha$ and $\beta$ denote the lengths of the segments $[u,w]$, and $[u',w']$, respectively.

We argue in a completely analogous manner, and establish first that
\[
  \vert P \vert \geqslant \kappa (\alpha + \beta  -4k + 4)
\]
and then that
\[
 \vert P \vert \geqslant \kappa (\alpha + \beta  -4k + 4) + \frac {n - 4k}N
\]
We then observe that $\vert h \vert = \alpha + \beta + 6k$, and the isoperimetric inequality leads to a bound on $n$ as before.
\end{proof}

Combining Theorem 3.3 in the case $n=2$ and Proposition 4.1 yields the immediate consquence that if $G$ is a $\PD{2}{R}$-group (where, as usual, $R$ is a $*$-ring and either $R$ is commutative or $G$ is orientable) then either $G$ is amenable or $G$ is hyperbolic. 

Now, subgroups of hyperbolic groups are either virtually cyclic or contain non-abelian free subgroups (see \cite{GdlH90}*{Chapter 8, Theorem 37} for a careful explanation of this fact) and so the above observation implies a Tits alternative: namely that every subgroup of a $\PD 2 R$ group is amenable or contains a non-abelian free subgroup. Moreover these two possibilities are clearly mutually exclusive, making an attractive contrast of the case of $\PD{1}{R}$-groups which are always virtually infinite cyclic and in particular both amenable and (elementary) hyperbolic.
We summarize these deductions:

\begin{cor}[Tits alternative]
\label{tits}
Let $G$ be a $\PD{2}{R}$-group as above. Then exactly one of the following holds:
\begin{enumerate}
\item $G$ is amenable.
\item $G$ is a non-elementary hyperbolic group.
\end{enumerate}
\end{cor}

\begin{rmk}
 \cref{tits} becomes false if one relaxes the assumption on the group from being a Poincar\'e duality group to being only a duality group with formal dimension $2$: for example free groups of finite rank $\geqslant 1$ are $1$-dimensional duality groups and so if $E$ and $F$ are non-trivial free groups of finite positive rank at least one of which has rank $\geqslant 2$ then $E\times F$ is a $2$-dimensional duality group (by an application of \cite{Bieri1981}*{Theorem 9.10}) and it is neither amenable nor hyperbolic.
 \smallskip
 
It is an interesting problem if one could find other examples of duality groups of formal dimension $2$, which are neither amenable nor hyperbolic. The realm of arithmetic groups seemed initially promising, due to the work of Borel--Serre  \cite{BorelSerre1970}. Unfortunately, the only examples the authors were able to find there were arithmetic lattices in $\mathrm{SO}(2,2)$ of $\Q$-rank $2$. Since $\mathrm{SO}(2,2)$ is isogenous to $\PSL_2(\R) \times \PSL_2(\R)$, and since irreducible lattices of $\Q$-rank at least two contain finite index subgroups of $\SL(3,\Z)$ or $\mathrm{SO}(2,3)_{\Z}$ (see \cite{WitteMorris2015}*{Theorem 9.2.7}), it is immediate that the lattices we found are reducible and arithmetic, and hence commensurable to $\PSL_2(\Z) \times \PSL_2(\Z)$ (\cite{WitteMorris2015}*{Propositions 4.3.3 and 6.1.5}), which is itself commensurable to a product of free groups.
\end{rmk}

\begin{rmk}
\cref{tits} also fails for $\PD{n}{R}$-groups  for any $n\ne2$. For $n\geqslant 3$ there always exist $\PD{n}{R}$-groups that are both non-amenable and non-hyperbolic essentially because such groups can admit embedded incompressible tori.
\end{rmk}

\medskip With the next two propositions we examine the amenable and hyperbolic cases in turn. We start with the amenable case, where we assume $R= \Z$.
The following is an adaption of the work of Degrijse~\cite{Degrijse2016}.

\begin{prop}[Amenable case]
\label{amen case}
Suppose that $G$ is an amenable $\PD 2 \Z$-group. Then $G$ is isomorphic to $\Z \rtimes \Z$.
\end{prop}
\begin{proof}
As explained by Eckmann~\cite{Eckmann1996}*{Section 4.1}, the usual complex Euler characteristic $\chi(G)$ of $G$ coincides with its $L^2$-Euler characteristic.
Note that $G$ is non-trivial, since the trivial group is not a $\PD 2 \Z$ group. We also know that $G$ is torsion-free, and hence it is infinite.
The $L^2$-Euler characteristic of infinite amenable groups vanishes by \cite{Lueck2002}*{Theorem 7.2(1)} of L\"uck.

We have $H_0(G;\C) = \C$. %Also, we have $H_2(G;\Z ) \neq 0$, since we have the fundamental class.
But, since $G$ is of type $\typeFP{2}$  over $\Z$, we have
\[
 0 = \chi(G) = \dim_\C  H_0(G;\C) + \dim_\C H_2(G;\C) - \dim_\C H_1(G;\C)
\]
and so $H_1(G;\C)$ has rank at least $1$. Since $G$ is finitely generated, this implies the existence of a non-trivial homomorphism $G \to \Z$; let $K$ denote  its kernel. Strebel's result~\cite{Strebel1977} tells us that $K$ is of $\Z$-cohomological dimension at most $1$, and so it is free by a result of Stallings~\cite{Stallings1968}. It is also amenable, and therefore is either trivial or isomorphic to $\Z$. In the former case we have $G = \Z$, which is not a $\PD 2 \Z$ group. In the latter case we have proven the claim.
\end{proof}

\iffalse
\begin{prop}[Amenable case]
Suppose that $G$ is an amenable $\PD 2 \Z$-group. Then $G$ is isomorphic to $\Z \rtimes \Z$.
\end{prop}
\begin{proof}
By a result of Berrick--Chaterjee--Mislin \dawid{cite}, the Bass conjecture (over $\Z$) holds for $G$, and therefore, as explained by Eckmann \dawid{cite}, the usual Euler characteristic $\chi(G)$ of $G$ coincides with its $L^2$-Euler characteristic. The latter vanishes (see \dawid{luck}) as $G$ is amenable.

Note that $G$ is torsion-free, since it is of finite cohomological dimension over $\Z$. Also, $G$ is non-trivial, since the trivial group is not a $\PD 2 \Z$-group.

Since $G$ is non-trivial, we have $H_0(G;\Z) = \Z$. %Also, we have $H_2(G;\Z ) \neq 0$, since we have the fundamental class.
But, since $G$ is of type $\typeFP{2}$  over $\Z$, we have
\[
 0 = \chi(G) = \rk H_0(G;\Z) + \rk H_2(G;\Z) - \rk H_1(G;\Z)
\]
and so $H_1(G;\Z)$ has rank at least $1$. This implies the existence of a non-trivial homomorphism $G \to \Z$; let $K$ denote the its kernel. Strebel's result \dawid{cite} tells us that $K$ is of $\Z$-cohomological dimension at most $1$, and so it is free by a result of Swan \dawid{cite}. It is also amenable, and therefore is either trivial or isomorphic to $\Z$. In the former case we have $G = \Z$, which is not a $\PD 2 \Z$-group. In the latter case we have proven the claim.
\end{proof}
\fi

For the hyperbolic case we need two preparatory lemmas.

\begin{lem}\label{partA}
Let $G$ be a group with a subgroup $H$ of finite index and a central subgroup $Z$ such that the following hold:
\begin{enumerate}
\item $Z\subseteq H$.
\item $G/Z$ is an orientable $\PD{2}{\Z}$-group. 
\item $Z$ has finite exponent dividing the index $|G:H|$.
\end{enumerate}
Then the extension $Z\rightarrowtail H\twoheadrightarrow H/Z$ is split and $H$ is isomorphic to $Z\times H/Z$. In particular, if $Z$ is finite then $G$ is virtually torsion-free.
\end{lem}
\begin{proof}
Write $\overline G=G/Z$ and $\overline H=H/Z$.
Initially we need only  that $\overline H$ has finite index in $\overline G$ and that $\overline G$ is an orientable $\PD{n}{\Z}$-group for some $n$.
 Since $\overline H$ has finite index in $\overline G$ there are both restriction and corestriction maps in both cohomology and homology: for any $\Z\overline G$-module $M$ we have natural commutative squares
\[
\xymatrix{
H^i(\overline G,M)\ar[r]^\simeq\ar[d]^{\mathrm{Res}}
&H_{n-i}(\overline G,M)\ar[d]^{\mathrm{Tr}}
&H^i(\overline G,M)\ar[r]^\simeq
&H_{n-i}(\overline G,M)\\
H^i(\overline H,M)\ar[r]^\simeq
&H_{n-i}(\overline H,M)
&H^i(\overline H,M)\ar[r]^\simeq\ar[u]_{\mathrm{Tr}}
&H_{n-i}(\overline H,M).\ar[u]_{\mathrm{Cor}}\\
}
\]
Here the horizontal maps are the Poincar\'e duality isomorphisms.
We refer to the corestriction in cohomology and restriction in homology as transfer maps. When the module $M$ is acted on trivially by $G$ then these transfer maps are both given by multiplication by the index $|G:H|$ in dimension $0$. Therefore, when $n=2$, $i=0$, and $M=Z$ the left hand square reduces to
\[
\xymatrix{
H^2(\overline G,Z)\ar[r]^\simeq\ar[d]^{\mathrm{Res}}
&H_{0}(\overline G,Z)\ar[d]^{\times|G:H|}\\
H^2(\overline H,Z)\ar[r]^\simeq
&H_{0}(\overline H,Z).\\
}
\]
Since $H$ has index equal to a multiple of the exponent of $Z$, the transfer map here vanishes and so the extension class in $H^2(\overline G,Z)$ representing $Z\rightarrowtail G\twoheadrightarrow\overline G$ vanishes on restriction to $\overline H$ and the desired splitting follows.
\end{proof}

We will follow Gabai~\cite{Gabai1992} and use the terminology \emph{Fuchsian group} to refer to any discrete subgroup of $\PGL_2(\R)$. Note that some authors prefer to consider only discrete subgroups of $\PSL_2(\R)$, which we will call \emph{orientable Fuchsian groups}. Observe that every Fuchsian group has a canonical orientable Fuchsian subgroup of index at most $2$ obtained by intersecting the group with $\PSL_2(\R)$. It is well known that every finitely generated orientable Fuchsian group $G$ is either virtually free (which includes the finite case), or virtually an \emph{orientable closed Fuchsian group}, that is, the fundamental group of a closed orientable surface of genus at least $2$ (see, e.g.,  \cite{Hoareetal1972}*{Page 60}). Moreover, if $G$ is torsion free, then $G$ is itself either free or an orientable closed Fuchsian group (this follows from \cite{BurnsSolitar1983}*{Lemma 2}). A Fuchsian group whose canonical orientable Fuchsian subgroup is a closed orientable Fuchsian group will be called a \emph{closed Fuchsian group}.

The group $\PGL_2(\R)$ acts on the Riemann sphere $\C \cup \{\infty\}$ via M\"obius transformations, and the action preserves the circle $\R \cup \{\infty\}$. This way we obtain a canonical action of every Fuchsian group on the circle. Moreover, this action is orientation preserving if and only if the Fuchsian group is orientable.

\begin{lem}\label{partB}
Let $G$ be a finitely generated group with a finite normal subgroup $K$ such that $G/K$ is Fuchsian. Then $G$ is virtually Fuchsian.
\end{lem}

\begin{proof} If $G/K$ is virtually a free group then $G$ has a finite index subgroup $H$ containing $K$ such that $H/K$ is free, and therefore the extension $K\rightarrowtail H\twoheadrightarrow H/K$ is split and $H$, and hence $G$, is virtually free.

Now suppose that $G/K$ is virtually an orientable closed Fuchsian group. Again, we may replace $G$ by a finite index subgroup and assume that $G/K$ is a closed orientable Fuchsian group.
Since $K$ is finite and normal, its centraliser $C_G(K)$ has finite index in $G$. Replacing $G$ by the centraliser $C_G(K)$ of $K$ and $K$ by $K\cap C_G(K)$, we may assume that $K$ is central in $G$. Note that the new quotient $G/K$ is a finite index subgroup of an orientable closed Fuchsian group, and hence is itself an orientable closed Fuchsian group.

Closed orientable Fuchsian groups are $\PD{2}{\Z}$-groups, so \cref{partA} can be applied to the central extension $K\rightarrowtail G\twoheadrightarrow G/K$. Since $G/K$ has subgroups of any finite index we can choose a subgroup of index equal to the exponent of $K$. By \cref{partA} we obtain a split central extension and find a subgroup of finite index with the desired properties.
\end{proof}

\begin{prop}[Hyperbolic case]
\label{hyp case}
Every hyperbolic oriented $\PD 2 R$ group $G$ contains a closed Fuchsian group as a  finite index subgroup.

If additionally $G$ is torsion free, then $G$ is isomorphic to a closed Fuchsian group, and this closed Fuchsian group is orientable if and only if the action of $G$ on the circle forming its Gromov boundary is orientation preserving. 
\end{prop}
\begin{proof}
 By \cite{Bestvina1996}*{Theorem 2.8 and Remark 2.9} of Bestvina, the Gromov boundary of $G$ is a closed $1$-manifold, that is, the circle. The action of a Gromov hyperbolic group on its boundary is a convergence action by a result of Bowditch~\cite{Bowditch1999}*{Lemma 2.11}.
 
 The action of $G$ (when $G$ is non-elementary, which is the case here) on its boundary has finite kernel $K$. This is a folklore result, which can be proven as follows: the boundary has at least three distinct points \cite{GdlH90}*{Chapter 7, Proposition-Definition 15}, and if an element of $G$ fixes more than two points in the boundary, then it is elliptic \cite{GdlH90}*{Chapter 8, Theorems 16 and 17} and hence of finite order \cite{GdlH90}*{Chapter 8, Proposition 28}. Now every torsion subgroup of a hyperbolic group is finite~\cite{GdlH90}*{Chapter 8, Corollary 36}.
 
The convergence theorem of Tukia, Gabai, and Casson--Jungreis \cites{Tukia1988,Gabai1992,CassonJungreis1994} informs us that $G/K$ is isomorphic to a Fuchsian group in a way taking the action of $G/K$ on its boundary to the canonical action of the Fuchsian group on the circle.
We now apply \cref{partB}, and conclude that $G$ is virtually Fuchsian (and finitely generated). Since virtually free groups are not $\PD 2 R$ as their $R$-cohomological dimension is either $1$ or infinity, we see that $G$ is virtually a closed Fuchsian group.

\smallskip
Now suppose that $G$ is torsion free. Then $K$ is trivial, and so $G$ is torsion free, finitely generated, and Fuchsian. It follows that $G$ is free or a closed Fuchsian group. But $G$ is $\PD 2 R$, and hence it cannot be free.

The canonical action of a Fuchsian group on its boundary is orientation preserving if and only if the group is a subgroup of $\PSL_2(\R)$, that is, if and only if it is an orientable Fuchsian group.
 \end{proof}

\begin{thm}
Every orientable $\PD 2 \Z$ group $G$ is isomorphic to the fundamental group of a closed orientable surface of positive genus.
\end{thm}
\begin{proof}
Observe that $G$ is torsion free, since it has finite cohomological dimension over $\Z$.
We use \cref{tits} to conclude that $G$ is either amenable or non-elementary hyperbolic. In the first case, \cref{amen case} tells us that $G$ is isomorphic to $\Z \rtimes \Z$. If the semi-direct product  structure is non-trivial, then the group is a non-orientable $\PD 2 \Z$ group, and in particular it is not an orientable $\PD 2 \Z$ group. Hence $G \cong \Z^2$. If $G$ is hyperbolic, then \cref{hyp case} tells us directly that $G$ is an orientable closed Fuchsian group, provided that we can show that the action of $G$ on its Gromov boundary $\partial G$ is orientation preserving. This is the case, since Bestvina~\cite{Bestvina1996}*{Proposition 1.5} gives us a $\Z G$-module isomorphism $H^1(\partial G; \Z) \cong H^2(G;\Z G)$, and hence $G$ being an orientable $\PD 2 \Z$ group translates directly into the action of $G$ on $H^1(\partial G; \Z)$ being trivial, which is the case precisely when the action of $G$ on $\partial G$ is orientation preserving.
\end{proof}

%\nocite{Zieschang_et_al,Serre2003,Freden1995,FarbMargalit2012,Gersten1996,KapovichKleiner2004}

%\bibliographystyle{math}
\bibliography{bibpd2}

\end{document}